\documentclass[12pt]{amsart}

\usepackage{amsmath,amssymb}

\usepackage{graphicx}
\usepackage{fancyhdr} 
\usepackage{fancyhdr}
\usepackage{fancyvrb}
\usepackage{tikz}
\usepackage{pgfplots}
\usepackage{amsmath}
\usepackage{flexisym}
\usepackage{breqn}
\usepackage{amsthm}
\usepackage{bbm}
\usepackage{url}
\usepackage{enumerate}
\usepackage{hyperref}
\usepackage{cleveref}
\usepackage{multicol}
\usepackage[T1]{fontenc}

\newcommand{\Q}{\mathbb{Q}}
\newcommand{\Z}{\mathbb{Z}}

\newcommand{\eps}{\varepsilon_{D}}
\newcommand{\epst}{\varepsilon_{2}}
\newcommand{\epsp}{\varepsilon_{2D}}
\newcommand{\kbar}{K^{\prime}}

\newcommand{\Qext}[1]{\mathbb{Q} ( #1 )}
\newcommand{\Qextbi}[1]{\mathbb{Q} \bl #1 \br}

\newcommand{\nclassgroup}[1]{Cl^{+}(#1)}

\newcommand{\bl}{\left(}
\newcommand{\br}{\right)}
\newcommand{\al}{\left|}
\newcommand{\ar}{\right|}
\newcommand{\pl}{\left[}
\newcommand{\pr}{\right]}

\newtheorem{theorem}{Theorem}
\numberwithin{theorem}{section}
\newtheorem*{theorem*}{Theorem}
\newtheorem{corollary}[theorem]{Corollary}

\newtheorem{proposition}[theorem]{Proposition}

\newtheorem{conj}[theorem]{Conjecture}

\newtheorem{lemma}[theorem]{Lemma}
\theoremstyle{plain}
\theoremstyle{definition}

\newtheorem{remarkref}[theorem]{Remark}

\begin{document}
\include{package}
\baselineskip=14.5pt

\title{ Iwasawa module of the cyclotomic $\Z_{2}$-extension of certain real quadratic fields}

\author{Josu\'e \'Avila}
\address{Centro de Investigaci\'on de Matem\'atica Pura y Aplicada (CIMPA), Universidad de Costa Rica.}
\email{josue.avila@ucr.ac.cr}
\renewcommand{\thefootnote}{}


\maketitle


\begin{abstract} For a real quadratic field $K=\Qext{\sqrt{D}}$, let $K_{\infty}$ denote the cyclotomic $\Z_{p}$-extension of $K$. Greenberg conjectured that the corresponding Iwasawa module $X_{\infty}$ is finite. Building on the work of Mouhib and Movahhedi, we provide new examples of real quadratic fields for which the conjecture holds, when $X_{\infty}$ is cyclic and the prime is $p=2$. Furthermore, we find a fundamental system of units for certain biquadratic fields of the form $\Qext{\sqrt{2}, \sqrt{D}}$ and show how to use it to calculate the order of $X_{\infty}$.
\end{abstract}

\section{Introduction}

\smallskip

Let $K=\Qext{\sqrt{D}}$ be a real quadratic field and $K_{\infty}$ denote the cyclotomic $\Z_2$-extension of $K$. Let $X_{\infty}$ be the Iwasawa module corresponding to this extension \cite{washington}. Greenberg conjectured \cite{greenberg} that (independent of the prime number) $X_{\infty}$ is finite whenever the base field is totally real. In (\cite{mouhib}, Theorem 3.8), Mouhib and Movahhedi classified all real quadratic fields such that the Iwasawa module $X_{\infty}$ is cyclic and non-trivial. Furthermore, they proved that the conjecture is satisfied for certain cases. In (\cite{laxmi}), Laxmi and Hariharan gave new examples of infinite families with $X_{\infty} \cong C_{2}$, the cyclic group of order $2$. Motivated by these two articles, we provide new examples of real quadratic fields where $X_{\infty}$ is cyclic and finite. Some of the techniques we will use can be found in (\cite{azizi2000}, \cite{azizi2001}, \cite{laxmi}, \cite{mouhib}). Moreover, one of our main references is the analysis of biquadratic forms done by Kaplan \cite{kaplan}, where the author uses Gauss' \cite{gaussoriginal} genus theory to determine the $4$-rank and $8$-rank of real quadratic fields depending on the conditions of the primes dividing the discriminant, as it will be explained later.\\

We will focus our attention on the following three cases where $X_{\infty}$ is cyclic (\cite{mouhib}, Theorem 3.8):
\begin{enumerate}\label{casesfirst}
    \item $D=pq$, with $p \equiv 5 \mod{8}$ and $q \not \equiv 1 \mod{8}$.
    \smallskip
    \item $D = p q_{1} q_{2}$ with $p \equiv 5 \mod{8} $, $q_{1} \equiv 3 \mod{8}$ and $q_{2} \equiv 3 \mod{4}$.
    \smallskip
    \item $D=pq$, with $p \equiv 5 \mod{8}$ and $q \equiv 1 \mod{8}$, with $\left( \frac{2}{q} \right)_{4} \neq (-1)^{\frac{p-1}{8}}$.
\end{enumerate}

In the first case, it was shown by Ozaki and Taya \cite{ozakitaya} that $X_{\infty}$ is finite, using Iwasawa theory, but the module was not determined explicitly. We prove that for  $q \equiv 3 \mod{4}$ the Iwasawa module $X_{\infty}$ is isomorphic to $A_{0}$ (\Cref{casep5q3}), whilst for $q \equiv 5 \mod{8}$ we prove that this holds under certain conditions (\Cref{caseqequal5}).\\

In the second case, when $q_{2} \equiv 7 \mod{8}$ the Iwasawa module $X_{\infty}$ is known to be finite \cite{mouhib}. Moreover, in \cite{laxmi} the authors proved that $X_{\infty} \cong A_{0} \cong \Z/2\Z$ under certain restrictions for both $q_{2} \equiv 7 \mod{8}$ and $q_{2} \equiv 3 \mod{4}$, using the properties of the biquadratic field $K_{1}$ and the units of the ring of integers.  Using our methods, we will provide a new family considering the field $K_{2}$ instead (\Cref{doublesymbol-1part2}).\\

Finally, in the third case, we use the classification in \cite{scholz} (see \Cref{scholz}) and the ideas in \cite{mouhib}, to provide new infinite families of real quadratic fields satisfying Greenberg's conjecture (\Cref{pqcase}). In addition to this, we give a Fundamental System of Units (F.S.U) of the field $K_{1}=\Q(\sqrt{2},\sqrt{D})$ for some of the cases (\Cref{corol1}, \Cref{corol2}, \Cref{lastcase1mod16}), using the analysis done by Kubota \cite{kubota}, and we draw some conclusions regarding the fundamental units of the quadratic subfields in this extension.\\

In the next section, we provide the main notations used throughout this article, and summarize some background results for the reader's convenience. The main results of this article will be discussed in sections $3, 4$, and $5$.

\section{Notations and preliminary results}

Denote by $K_n$ the $n^{\rm{th}}$ layer of the cyclotomic $\Z_2$-extension of $K=\Q(\sqrt{D})$ and $A_n$ its $2$-class group. These fields are of the form $K_n=\Q(a_{n},\sqrt{D})$, where $a_0 = 0$ and $a_n = \sqrt{ 2 + a_{n-1}}$. Since they will be of great importance, we emphasize that $K_0 = K$, $K_1 = \Q(\sqrt{2}, \sqrt{D})$ and $K_{2}=\Q \left( \sqrt{2+\sqrt{2}},\sqrt{D}\right)$. We denote by $n_{0}$ the first integer such that the extension $K_{\infty}/K_{n_{0}}$ is totally ramified \cite{washington}. Also, let $X_{\infty} = \lim\limits_{\substack{ \longleftarrow \\ }} A_{n}$ be the Iwasawa module corresponding to this extension.\\

Denote by $\Q_n$ the $n^{\rm{th}}$ layer of the cyclotomic $\Z_2$-extension of $\Q$. In particular $\Q_{1} = \Q(\sqrt{2})$ and its fundamental unit is $\epst=1+\sqrt{2}$ of norm equal to $-1$. Following \cite{mizusawa}, we shall also let $K_{n}^{\prime}$ be the subfield of $K_{n+1}$ containing $\Q_n$, different from $K_n$ and $\Q_{n+1}$, and $A_{n}^{\prime}$ its $2$-class group. Of particular importance, we have $K_{0}^{\prime}=\kbar = \Q( \sqrt{2D})$ and $K_{1}^{\prime}=\Q(\sqrt{(2+\sqrt{2})D}$. We denote the fundamental units of $K$ and $\kbar$ by $\eps$ and $\epsp$ respectively.\\

As it will be one of our main references, we adhere to the original notation from Gauss \cite{gaussoriginal} used by Kaplan. The quadratic form $Ax^2+2Bxy+Cy^2$ will be denoted by $[A, B, C]$, where the condition regarding the second coefficient is the main difference from the modern definition. As established in \cite{kaplan}, there is a map between the group $C(D)$ of equivalence classes of quadratic forms of discriminant $D = B^2-AC$, using this restriction, and the $2$-part of the narrow class group of $K=\Q(\sqrt{D})$, where the kernel is a group of order $1$ or $3$. Since we will mainly focus on the $2$-part of a group, there will be no difference in the overall theory.\\

For a finite abelian group $G$, we can define the $2^n$-rank of $G$ as 
\begin{align*}
    r_{2^n}(G) = \text{dim}_{\mathbf{F}_{2}}\bl G^{2^{n-1}}/G^{2^{n}} \br,
\end{align*}
where $\mathbf{F}_{2}$ represents the finite field of $2$ elements. A result due to Gauss \cite{gaussoriginal} is that the $2$-rank of the narrow class group of $K$, denoted by $r_{2}$ for simplicity, satisfies $r_{2}=r-1$,
 where $r$ is the number of divisors of the discriminant $d_K$ of $K$. The analysis for the $4$-rank and $8$-rank is more complicated. \\
 
 Two classes are in the same genus if their generic characters \cite{kaplan} have the same values. We define the principal genus as the genus of the principal form $[1, 0, -D]$. In particular, a class is in the principal genus if the value of every generic character equals $1$. Similar to the imaginary case \cite{cox}, let $s$ be the number of ambiguous simple forms in the principal genus, then 
 \begin{align*}
     s = 2^{\; r_4+1},
 \end{align*}
 where $r_4$ is the $4$-rank of the group $\nclassgroup{K}$. This will be useful in our analysis later, but we will refer to Kaplan's paper when needed since he determines the $4$-rank and $8$-rank of the fields of interest. \\

For future reference, we will write down some of the main results that will be used throughout the article.  Any result that is left out of this section will be included in its respective section. First, we have the following result due to Fukuda:

\begin{theorem} \cite{fukudamain} \label{fukuda stability}
Let $K_{\infty}/K$ be a $\Z_{p}$ extension such that $K_{\infty}/K$ is totally ramified. If $\al A_1 \ar = \al A_{0} \ar$, then $\al A_{n}\ar=\al A_{0}\ar$ for all $n \geq 0$.
\end{theorem}
For the cyclotomic extension of a totally real field, since any field $K_{n}$ in the chain has the same cyclotomic extension, this can be simply stated as $X_{\infty} \cong A_{n}$ as soon as $\al A_{n+1} \ar = \al A_{n} \ar$ for some $n \geq n_{0}$.\\

We also have the following result about fundamental units for real biquadratic fields:

\begin{theorem} \cite{kubota} \label{kubotakuroda formula}
Let $L/\mathbb{Q}$ be a totally real bi-quadratic extension, with unit group $E(L)$ and $2-$class group $A(L)$. Let $L_1, L_2$ and $L_3$ be the quadratic subfields of $L$. Let $\varepsilon_i$ be the fundamental unit of $L_i$, for $i=1,2$ and $3$. Let $Q(L) := [E(L) : \langle -1, \varepsilon_{1}, \varepsilon_{2}, \varepsilon_{3} \rangle]$ be the Hasse unit index of $L$. Then we have
\begin{align*}
|A(L)| = \dfrac{1}{4}\cdot Q(L)\cdot |A(L_1)|\cdot |A(L_2)|\cdot |A(L_3)|.
\end{align*}
Further, a system $\mathcal{F}$ of fundamental units of $L$ (F.S.U) is one of the following possibilities.
\begin{multicols}{2}
\begin{enumerate}
    \item $\{\varepsilon_1, \varepsilon_2, \varepsilon_3 \}$
    \item $\{\sqrt{\varepsilon_1}, \varepsilon_2, \varepsilon_3 \}$
    \item $\{ \sqrt{\varepsilon_1}, \sqrt{\varepsilon_2}, \varepsilon_3\}$
    \item $\{\sqrt{\varepsilon_1\varepsilon_2}, \varepsilon_2, \varepsilon_3 \}$
    \item $\{\sqrt{\varepsilon_1\varepsilon_2}, \varepsilon_2, \sqrt{\varepsilon_3} \}$
    \item $\{ \sqrt{\varepsilon_1\varepsilon_2}, \sqrt{\varepsilon_1\varepsilon_3}, \sqrt{\varepsilon_2\varepsilon_3}\}$
    \item  $\{ \sqrt{\varepsilon_1\varepsilon_2\varepsilon_3}, \varepsilon_2, \varepsilon_3 \}$
\end{enumerate}
\end{multicols}
Any unit $\varepsilon_{i}$ that appears under the square root is assumed to have norm equal to 1, except for the last case, where all the units can have the same norm, either all 1 or all -1.
\end{theorem}

Finally, for the first and third cases of our results, we have the following classification due to Scholz: 

 \begin{theorem} \label{scholz}
(\cite{scholz}, \textbf{cf.}, \cite{kaplandiv} Proposition 2.2), Let $N$ denote the norm map from $K$ to $\Q$, and $h$, $h^{+}$ denote respectively the class number and the narrow class number of $K$. Assume $p \equiv q \equiv 1 \mod{4}$. The following statements hold:
 \begin{enumerate} 
\item If $\left( \dfrac{p}{q} \right)=-1$, then $h^+ \equiv h \equiv 2 \pmod 4$ and $N(\eps)=-1$.
 \item If $\left( \dfrac{p}{q} \right)=1$, then:
 \begin{enumerate}[i-]
\item If  $\left( \dfrac{p}{q} \right)_{4}=-\left( \dfrac{q}{p} 
\small\right)_{4}$, then $h^+ \equiv 2h \equiv 4 \pmod 8$ and $N(\eps)=1$.
\item If  $\left( \dfrac{p}{q} \right)_{4}=\left( \dfrac{q}{p} \right)_{4}=-1$, then $h^+ \equiv h \equiv 4 \pmod 8$ and $N(\eps)=-1$.
\item If  $\left( \dfrac{p}{q} \right)_{4}=\left( \dfrac{q}{p} \right)_{4}=1$, then $h^+ \equiv 0 \pmod 8$.
\end{enumerate}
  \end{enumerate}
  \end{theorem}

\section{First Case}

Throughout this section, we focus on the case when $D=pq$ with $p \equiv 5 \mod{8}$ and $q \not \equiv 1 \mod{8}$. As mentioned before, the corresponding Iwasawa module is finite (\cite{ozakitaya}, \textbf{cf} \cite{mouhib}). First, we assume that $q \equiv 3 \mod{4}$. Under these conditions, we show that the Iwasawa module is isomorphic to $A_{0} \cong C_{2}$, where $C_{n}$ denotes the cyclic group of order $n$:

 \begin{proposition} \label{casep5q3}
 Let $K=\Q(\sqrt{D})$ with $D=pq$, $p \equiv 5 \mod{8}$ and $q \equiv 3 \mod{ 4}$. Then $\{ \sqrt{\eps \epsp}, \eps, \epst \}$ is a F.S.U. of $K_{1}$ and $X_{\infty} \cong A_{0} \cong C_{2}$.
\end{proposition}
\begin{proof}
The simple ambiguous forms of $K$ are
\begin{align*}
    f=[1,0,-pq], \; g=[p,0,-q], \; h=\pl 2, 1, \frac{1-pq}{2} \pr, \; l= \pl 2p,p,\frac{p-q}{2} \pr, \\
    \overline{f}=[-1,0,pq], \; \overline{g}=[q,0,-p], \; \overline{h}=\pl -2, 1, \frac{pq-1}{2} \pr, \; \overline{l}= \pl 2q,q,\frac{q-p}{2} \pr ,
\end{align*}
where exactly only one of the forms $g$, $l$ or $\overline{l}$ is in the principal genus, and the narrow class group of $K$ is isomorphic to $C_{2} \times C_{2}$ (\cite{kaplan}, Proposition $A_{3}$). Since $N(\eps)=1$, the class number differs by a factor of 2. Therefore, $A_{0} \cong C_{2}$.\\ 

Without loss of generality, assume that the form $l$ is in the principal genus. The principal genus has only two forms, which implies that $l$ is equivalent to the unit $f = [1, 0, -pq]$. Furthermore, since $l$ represents the number $2p$, so does $f$. Therefore, the equation  $x^2-pqy^2=2p$ has an integral solution. Since both $2$ and $p$ ramify in $K$, the ideal generated by $2p$ factorizes as $\langle 2p \rangle = \mathcal{L}^{2}$ for some ideal $\mathcal{L}$ of $K$. The ideals $\langle x-y \sqrt{pq} \rangle$, $\langle x+y \sqrt{pq} \rangle$ are conjugates and the ideals over $2$ and $p$ are self-conjugates. It follows by unique factorization that the ideal $\mathcal{L}$ is principal. \\

Consequently, we have $\langle 2p \rangle = \langle z \rangle ^2$ for some fractional ideal $\langle z\rangle$ of $K$. Therefore, $2p = z^2 \cdot \eps^n$ for some n.
If $n$ is even, then $\sqrt{2p} \in K_{1}$, which is impossible, since $ K_{1} = \Qext{\sqrt{2}, \sqrt{pq}}$. Thus, $n$ is odd. Let $\alpha = |z| \cdot \eps^{(n-1)/2}$, then $\sqrt{2}\sqrt{p} = \alpha \cdot \sqrt{\eps}$,
which implies that $K_{1}(\sqrt{p}) = K_{1}(\sqrt{\eps}) \neq K_{1}$. Since $K_{1}(\sqrt{p})=K_{1}(\sqrt{q})$, the other cases follow in a similar manner. In any case, we can conclude that $\sqrt{\eps} \not \in K_{1}$.\\

Following the same lines, due to Kaplan (\cite{kaplan}, §6, Proposition $A_{4}$), the narrow $2$-class group of $K^{\prime} = \Qext{\sqrt{2D}}$ is isomorphic to $C_{2} \times C_{2}$. Since $N(\epsp)=1$, we can conclude that  $A_{0}^{\prime} \cong C_{2}$.  By analyzing the generic characters of $K^{\prime}$, we can observe that exactly one of the forms $\left [-p, 0, 2q \right ]$,  $\left [q, 0, -2p \right ]$ or $\left [-q, 0, 2p \right ]$ is equivalent to the unit form $f = [1, 0, -2pq]$.  Hence, the ideal over one of the elements $2q$, $2p$, or $q$ is principal in $\kbar$ and $K_{1}\bl \sqrt{p} \br=K_{1} \bl \sqrt{\epsp} \br$, reasoning like before. Therefore, $K_{1}(\sqrt{\eps})=K_{1}(\sqrt{\epsp})$, which implies that $\sqrt{\eps \epsp} \in K_{1}$.\\

By \Cref{kubotakuroda formula}, the F.S.U of $K_{1}$ is $\{ \sqrt{\eps \epsp}, \eps, \epst \}$ and the Hasse unit index $q(K_{1})=[E(K_{1}) : \langle -1, \eps, \epsp, \epst\rangle]$ is equal to $2$. Hence,
    $\al A_{1}\ar = \frac{1}{2} \al A_{0} \ar \cdot \al A_{0}^{\prime} \ar = \al A_{0} \ar.$ Since the extension $K_{\infty}/K$ is totally ramified (the prime $2$ is ramifies in every quadratic subextension of $K_{1}/\Q$), it follows by \Cref{fukuda stability} that $X_{\infty} \cong A_{0}$.
\end{proof}
\vspace{0.5cm}

The case $q \equiv 5 \mod{8}$ is more complicated. Using the classification in \Cref{scholz}, we prove the following result:\\

\begin{proposition} \label{caseqequal5} 
 Let $K=\Q(\sqrt{D})$ with $D=pq$ and $p \equiv q \equiv 5 \mod 8$. Then the following statements hold:
\begin{enumerate}
\item  Assume $\left( \dfrac{p}{q} \right)=1$ and $\left( \dfrac{p}{q} \right)_{4}=-\left( \dfrac{q}{p} \right)_{4}$. Then $\{ \eps, \epsp, \epst \}$ is a F.S.U of $K_{1}$ and $X_{\infty} \cong A_{0}$.
\item Assume either $\left( \dfrac{p}{q} \right)=-1$, or $\left( \dfrac{p}{q} \right)=1$ with $\left( \dfrac{p}{q} \right)_{4}=\left( \dfrac{q}{p} \right)_{4}=-1$.
Then
\begin{align*}
    X_{\infty} \cong A_{0} \iff \sqrt{\eps \epsp \epst} \not \in K_{1}.
\end{align*}
In this case, $\{ \eps, \epsp, \epst \}$ is a F.S.U of $K_{1}$. Otherwise, $\{ \sqrt{\eps \epsp \epst}, \eps, \epst \}$ is the corresponding F.S.U and $\al A_{1}\ar = 2 \cdot |A_{0}|$.
\vspace{0.1cm}
\item  Assume $\left( \dfrac{p}{q} \right)=1$ and $\left( \dfrac{p}{q} \right)_{4}=\left( \dfrac{q}{p} \right)_{4}=1$. 
\vspace{0.1cm}
\begin{enumerate}
    \item[i-] If $N(\eps)=1$, then $\{ \eps, \epsp, \epst \}$ is a F.S.U of $K_{1}$ and $X_{\infty} \cong A_{0}$.
    \item [ii-] If $N(\eps)=-1$, then
\begin{align*}
    X_{\infty} \cong A_{0} \iff \sqrt{\eps \epsp \epst} \not \in K_{1}.
\end{align*}
In this case, $\{ \eps, \epsp, \epst \}$ is a F.S.U of $K_{1}$. Otherwise, $\{ \sqrt{\eps \epsp \epst}, \eps, \epst \}$ is the corresponding F.S.U and $\al A_{1}\ar = 2 \cdot |A_{0}|$.
\end{enumerate}
\end{enumerate}
\end{proposition}

\begin{proof}

 From (\cite{kaplan}, §4, Proposition $A_{2})$, the only form in the principal genus of $\kbar$ is $\overline{f}=[-1, 0, 2pq]$ and $A_{0}^{\prime} \cong C_{2} \times C_{2}$, independent of the conditions on the quadratic and biquadratic symbols. We can also deduce that $N(\epsp)=-1$.
\\

For the field $K$, only one of the forms $[-1, 0, pq]$, $[p,0,-q]$ or $[q, 0, -p]$ is equivalent to $\left[ 1,0,-pq \right]$, since these are the forms in the principal genus \cite{kaplandiv} and each class has exactly two simple representatives. If $N(\eps)=1$, then $-1$ is not represented by $[1, 0, -pq]$, so that one of $p$ or $q$ is represented by this form. We can then conclude that either the ideal over $p$ or $q$ in $K$ is principal and $K_{1}(\sqrt{p})=K_{1}(\sqrt{q})=K_{1}(\sqrt{\eps})$, using a similar argument as before. Hence, $\sqrt{\eps} \not \in K_{1}$. 
\\

If $\left( \dfrac{p}{q} \right)=1$ and $\left( \dfrac{p}{q} \right)_{4}=-\left( \dfrac{q}{p} \right)_{4}$, we can see that $N(\eps)=1$ due to \Cref{scholz}. Given that only the unit $\eps$ has norm $1$ and $\sqrt{\eps} \not \in K_{1}$, it follows from \Cref{kubotakuroda formula} that $\{ \eps, \epsp, \epst \}$ is the F.S.U. of $K_{1}$, the Hasse unit index $q(K_{1})=1$ and $\al A_{1}\ar = \al A_{0}\ar $. Since $K_{\infty}/K$ is totally ramified, we must have $X_{\infty} \cong A_{0}$ in this case. Now assume either $\left( \dfrac{p}{q} \right)=-1$, or $\left( \dfrac{p}{q} \right)=1$ with $\left( \dfrac{p}{q} \right)_{4}=\left( \dfrac{q}{p} \right)_{4}=-1$. From \Cref{scholz}, the norm of all the fundamental units is $-1$. Hence, the only possible fundamental systems for $K_{1}$ are $\{ \eps, \epsp, \epst \}$ or $\{ \sqrt{\eps  \epsp  \epst}, \epsp, \epst \}$. Therefore, $\al A_{1}\ar =q(K_{1})\al A_{0} \ar$, and $\al A_{1}\ar =\al A_{0}\ar $ if and only if  $ \sqrt{\eps  \epsp  \epst} \not \in K_{1}$. Finally, if $\left( \dfrac{p}{q} \right)=1$ with $\left( \dfrac{p}{q} \right)_{4}=\left( \dfrac{q}{p} \right)_{4}=1$, the result follows similarly, depending on whether $N(\eps)=1$ or not.
\end{proof}

\begin{remarkref}
The condition $ \sqrt{\eps  \epsp  \epst} \in K_{1}$ is explored in \cite{aziziequation}. Let $\eps = x+y\sqrt{D}$. Then from the congruence $x^2-y^2 \cong -1 \mod 4$, $x$ must be even and $y$ must be odd. Since $x^{2}+1 = Dy^2$, we have a factorization $(x+i)(x-i)=Dy^2$ in the unique factorization domain $\Z[i]$. Both primes split in $\Qext{i}$, therefore there exist primes $\pi_{1}$, $\pi_{2}$ of $\Z[i]$ such that $p = \pi_{1} \overline{\pi_{1}}$, $q = \pi_{2} \overline{\pi_{2}}$, where $\overline{\pi_{j}}$ denotes the complex conjugate, and these primes can be chosen to be primary.
\\

Following the calculations in (\cite{williams}, Theorem 1), we can assume that both $\pi_{1}, \pi_{2} $ divide $x+i$, changing to the conjugate if necessary. This choice turns out to be significant when $\left( \dfrac{p}{q} \right)=1$. The condition is then equivalent to 
\begin{align*}
     \bl \frac{1+i}{\pi_{1}} \br_{2} = \bl \frac{1+i}{\pi_{2}} \br_{2},
 \end{align*}

 where $\bl \frac{a}{b} \br_{2}$ represents the quadratic symbol in $\Z[i]$. This is given in \cite{aziziequation}, and in the case where $\left(\frac{p}{q}\right)=-1$, it is equivalent to $\Bigl( \frac{pq}{2} \Bigr)_{4} \bl \frac{2p}{q} \br_{4} \bl \frac{2q}{p} \br_{4} = -1$. The behavior when $ \sqrt{\eps  \epsp  \epst} \in K_{1}$ is more complex since it is hard to predict when the class number will stop growing.
 \end{remarkref}
 \section{Second Case}

 We now focus on the case when $D = p q_{1} q_{2}$ with $p \equiv 5 \mod{8} $, $q_{1} \equiv 3 \mod{8}$ and $q_{2} \equiv 3 \mod{4}$. Independent of $q_{2}$, we have the following result which will be helpful later in this section:

\begin{lemma} \label{k1hatsecondcase}
    Let $K=\Qext{\sqrt{D}}$ where $D = p q_{1} q_{2}$ with $p \equiv 5 \mod{8} $, $q_{1} \equiv 3 \mod{8}$ and $q_{2} \equiv 3 \mod{4}$. Then the 2-class group of $K_{1}^{\prime}=\Qextbi{\sqrt{(2+\sqrt{2})D}}$ is isomorphic to $ C_2 \times C_2$.
\end{lemma}

\begin{proof}
    We first show that the 2-rank of $A_{1}^{\prime}$ is $2$. Consider the extension $K_{1}/ \Q_{1}$. Then $\epst$ is not a norm from $K_{1}$ (\cite{azizi2000}, Theorem 1). Indeed, from (\cite{azizi2004}, Theorem 2.5), if $\mathcal{Q}$ is the ideal over $q_{1}$ in $\Q_{1}$ (there is only one), then the Hilbert symbol satisfies
    \begin{align*}
    \left( \dfrac{\epst,D}{\mathcal{Q}} \right) = \bl \dfrac{N_{\Q_{1}/\Q}(\epst),D}{q_{1}} \br = \bl \dfrac{-1,D}{q_{1}} \br = -1,
\end{align*}
    which implies that $\epst$ is not a norm, given that $K_{1} = \Q_{1}(\sqrt{D})$ \cite{azizi2001}.
\\

Following (\cite{mouhib}, Lemma 3.5), if $\mathcal{L}$ is not a 2-adic prime of $\Q_1$, then the extension $\Q_2 /\Q_1$ is unramified at $\mathcal{L}$. Since $\Q_{2}=\Q_{1}(\epst \sqrt{2})$, the Hilbert symbol satisfies:
\begin{align*}
    \left( \dfrac{\epst, \epst \sqrt{2}}{\mathcal{L}} \right)=1.
\end{align*}
Multiplying the symbols we obtain
\begin{align*}
    \left( \dfrac{\epst,D\epst \sqrt{2}}{\mathcal{Q}} \right)=-1.
\end{align*}

Thus, we know that $\epst$ is not a norm from $K_{1}^{\prime}=\Q_{1}(\sqrt{D\epst\sqrt{2}})$. Using a similar argument, one can conclude that $-1$ is not a norm if $q_{2} \equiv 7 \mod{8}$, since for a prime $\mathcal{Q}$ over $q_{2}$ satisfying $\left(\frac{2}{q_{2}}\right)=1$, Azizi \cite{azizi2000} showed that
\begin{align*}
    \left( \dfrac{-1,D}{\mathcal{Q}} \right)= \left(\frac{-1}{q_{2}}\right) = -1,
\end{align*}
whereas for $q_{2} \equiv 3 \mod{8}$, this symbol is always one.
\\

Next, we apply the genus formula \cite{genusformula} to the extension $K_{1}^{\prime}/\Q_{1}$. If $q_{2} \equiv 3 \mod{8}$, there are $4$ ramifying primes in this extension: the $3$ primes from the divisors of $D$, and one prime over $2$. Hence, $r_{2}(A_{0}^{\prime})=4-1-e=3-e$,
where $e = [E(\Q_{1}): N_{K_{1}^{\prime}/\Q_{1}}(K_{1}^{\prime \; *})]$ is the norm unit index. Thus, the $2$-rank is at most $2$ ($\epst$ is not a norm). Since the extension $K_{1}^{\prime} (\sqrt{p}, \sqrt{q_{1}q_{2}})$ is unramified and biquadratic over $K_{1}^{\prime}$, it follows that the $2$-rank of the class group is exactly $2$. When $q_{2} \equiv 7 \mod{8}$, there are $5$ primes that ramify, since $q_{2}$ splits in $\Q_{1}$. Given that $-1$ and $\epst$ are not norms, we must have $r_{2}(A_{0}^{\prime})=4-e=4-2=2.$ In any case, the class group of $K_{1}^{\prime}$ is of the form $C_{2^m} \times C_{2^n}$ for some $m,n \geq 1$.
\\

The cyclicity of the groups $A_{n}$ plays a significant role in the following argument. Consider the intermediate subfields  of the extension $M = K_{1}^{\prime} (\sqrt{p}, \sqrt{q_{1}q_{2}})$ over $K_{1}^{\prime}$; i.e, the fields
\begin{align*}
    M_{1} = \Qextbi{\sqrt{2+\sqrt{2}}, \sqrt{pq_{1}q_{2}}}=K_{2} \; ;\\
    M_{2} =\Qextbi{\sqrt{(2+\sqrt{2})p}, \sqrt{q_{1}q_{2}} } ;\\
    M_{3} =\Qextbi{\sqrt{(2+\sqrt{2})q_{1}q_{2}}, \sqrt{p} }.
\end{align*}

Suppose $H$ is the 2-class field of $K_{1}^{\prime}$ with Galois group $ C_{2^m} \times C_{2^n}$. If $m, n \geq 2$, the abelian extension $M/K_{1}^{\prime}$ is unramified and its Galois group $\text{Gal}(M/K_{1}^{\prime}) \cong C_{2} \times C_{2}$ is contained in both $C_{4} \times C_{2}$ and $C_{2} \times C_{4}$, which are proper subgroups of $\text{Gal}(H/K_{1}^{\prime})$. It follows that the two unramified extensions corresponding to these subgroups contain both $M$ and $M_{1}=K_{2}$, which is impossible since $K_{2}$ cannot have two different unramified abelian extensions of the same order, given that its class group is cyclic.
\\

Therefore, $A_{1}^{\prime} \cong C_{2} \times C_{2^n}$ for some $n \geq 1$. The extension $H/K_{2}$ is unramified and abelian, hence it must be cyclic. It follows that $Gal(H/M)$ is also cyclic, and it is contained in $Gal(H/M_{i})$ for $i=1,2,3$, which are the $3$ different subgroups of index two of $A_{1}^{\prime}$. Let $A_{1}^{\prime} \cong \langle a \rangle \times \langle b \rangle$, where $a$ has order $2$ and $b$ has order $2^n$. The subgroups of index $2$ correspond to $\langle b \rangle$, $\langle a \cdot b \rangle$, and $\langle  a \rangle \times \langle b^{2} \rangle$. The intersection of these three groups is trivial, which implies that $Gal(H/M)$ is trivial. Hence, $H=M$ and he result follows, since $Gal(M/K_{1}^{\prime}) \cong C_{2} \times C_{2}$.
\end{proof}
\vspace{0.2cm}

In \cite{laxmi}, the authors proved that $X_{\infty} \cong A_{0}$ in the following cases:
\begin{enumerate}
    \item $q_{2} \equiv 3 \mod 8$, where $\left(\frac{q_{1}q_{2}}{p} \right)=-1$.
    \item $q_{2} \equiv 3 \mod 8$, where $\left(\frac{q_{1}}{p} \right)=\left(\frac{q_{2}}{p} \right)=1$, and the ideal above $p$ in $K$ is not principal.
    \item $q_{2} \equiv 7 \mod 8$, where $\left(\frac{q_{2}}{p} \right)=-1$.
    \item $q_{2} \equiv 7 \mod 8$, where $\left(\frac{q_{2}}{p} \right)=\left(\frac{q_{1}}{p} \right)=1$, and the ideal above $q_{2}$ in $K$ is not principal.
\end{enumerate}

We focus on the case where $q_{2} \equiv 3 \mod{8}$ and $\left(\frac{q_{2}}{p} \right)=\left(\frac{q_{1}}{p} \right)=-1$ to produce a new type of field satisfying Greenberg's conjecture. First, we have the following result, which will give us a necessary and sufficient condition for our example:

\begin{lemma} \label{doublesymbol-1}
    Let $K = \Qext{\sqrt{D}}$ where $D=pq_{1}q_{2}$, such that $p \equiv 5 \mod{8}$, $q_{1}\equiv q_{2} \equiv 3 \mod{8}$, and $\left(\frac{p}{q_{1}}\right) =\left(\frac{p}{q_{2}}\right) = -1$. Then the ideal above $p$ in $K$ is principal. Furthermore, the only possible F.S.U. of $K_{1}$ are $\{\eps, \epsp, \epst \}$ or $\{ \sqrt{\eps\epsp}, \eps, \epst\}$.\\
    
    Let $\epsp=r+s\sqrt{2D}$. Then  $\sqrt{\eps \epsp} \in K_{1}$ if and only if $p(r+1)$ is a square. In this case, 
    \begin{align*}
    N_{K_{1}/K}(\sqrt{\eps \epsp})=-\eps,\; N_{K_{1}/K^{\prime}}(\sqrt{\eps \epsp})=\epsp, \; N_{K_{1}/\Q_{1}}(\sqrt{\eps \epsp})=-1.
    \end{align*}
\end{lemma}
\vspace{0.1cm}

\begin{proof}
The simple ambiguous forms of $K$ are:
\begin{align*}
    f=[1,0,-pq_{1}q_{2}], \; g=[p,0,-q_{1}q_{2}], \; h=[q_{1}, 0, -pq_{2}], \; l=[q_{2},0,-pq_{1}], \\
    \overline{f}=[-1,0,pq_{1}q_{2}], \; \overline{g}=[-p,0,q_{1}q_{2}], \; \overline{h}=[-q_{1}, 0,pq_{2}], \; \overline{l}=[-q_{2},0,pq_{1}].
\end{align*}
There are some minor typos in Kaplan's paper (\cite{kaplan}, Proposition $A_{5})$. Let $\left( \frac{p}{q_{1}} \right)=\tau$, $\left( \frac{p}{q_{2}} \right) = \tau^{\prime}$ and $\left( \frac{q_{1}}{q_{2}} \right) = \tau^{\prime \prime}$. Then the table of generic characters is: 
\renewcommand{\arraystretch}{2}
\begin{table}[h]
    \centering
\begin{tabular}{ |c|c|c|c|c|c|c|c|c| } 
\hline
 & $f$ & $\overline{f}$ & $g$ & $\overline{g}$ & $h$ & $\overline{h}$ & $l$ & $\overline{l}$\\
\hline
 $\left( \frac{m}{p} \right)$ & $1$ & $1$ & $\tau \tau^{\prime}$ & $\tau \tau^{\prime}$ & $\tau$ & $\tau$ & $\tau^{\prime}$ & $\tau^{\prime}$ \\ 
 $\left( \frac{m}{q_{1}} \right)$ & $1$ & $-1$ & $\tau$ & $-\tau$ & $\tau\tau^{\prime \prime}$ & $-\tau\tau^{\prime \prime}$ & $-\tau^{\prime \prime}$ & $ \tau^{\prime \prime}$ \\ 
 $\left( \frac{m}{q_{2}} \right)$ & $1$ & $-1$ & $\tau^{\prime}$ & $-\tau^{\prime}$ & $\tau^{\prime \prime}$ & $-\tau^{\prime \prime}$ & $-\tau^{\prime} \tau^{\prime \prime}$ & $ \tau^{\prime} \tau^{\prime \prime}$ \\ 
\hline
\end{tabular}

\end{table}

Hence, $\overline{g}$ is the only form in the principal genus (all the generic characters are $1$) and the $2$-part of the narrow class group is $C_{2} \times C_{2}$. Since the norm of the fundamental unit is $1$ (a prime congruent to $3$ modulo $4$ divides the discriminant), we must have $A_{0} \cong C_{2}$. Moreover, $q_{1}q_{2}$ is represented by $[1, 0, -D]$, which implies that $K_{1}(\sqrt{\eps})=K_{1}(\sqrt{q_{1}q_{2}})=K_{1}(\sqrt{p})$ and $\sqrt{\eps} \not \in K_{1}$. 
\\

 Replacing $\eps$ with $\eps^3$ if necessary, we can assume that $\eps=x+y\sqrt{D}$ where $x, y \in \Z$  with $x$ odd and $y$ even. Since $x^2-Dy^2=1$, due to the factorization $(x-1)(x+1)=Dy^2$, we must have
\begin{equation*}
    \begin{cases}
        x + 1 = 2d_{1} y_{1}^{2}, \\
        x - 1 = 2d_{2} y_{2}^{2},
    \end{cases}
\end{equation*}
with $d_{1}d_{2}=D$ and $2y_{1}y_{2}=y$. From \cite{azizi2000}, $d_{i}$ cannot be neither $1$ nor $D$. \\

Let $z = y_{1} \sqrt{2d_{1}}+y_{2}\sqrt{2d_{2}}$, then 
\begin{align*}
    z^2=(2d_{1}y_{1}^2+2d_{2}y_{2}^2+2y\sqrt{D})=2(x+y\sqrt{D})=2\eps.
\end{align*}
Hence, $z=\sqrt{2 \eps}$ and $ 2\sqrt{d_{1}\eps}=\sqrt{2d_{1}}z=2(y_{1}d_{1}+y_{2}\sqrt{D})=2\gamma$, where
$N_{K/\Q}(\sqrt{d_{1}\eps})=N_{K/\Q}(\gamma)=d_{1}(d_{1}y_{1}^2-d_{2}y_{2}^2)=d_{1}$. Notice that $N_{K/Q}(\sqrt{d_{2}\eps})=-d_{2}$.\\

In particular, since $K_{1}(\sqrt{\eps})=K_{1}(\sqrt{p})=K_{1}(\sqrt{d_{i}})$, we can assume that $d_{1}=p$ or $d_{1}=q_{1}q_{2}$.
If $d_{1}=p$, then
\begin{align*}
    1=\left(\frac{2p}{q_{1}} \right)=\left(\frac{2py_{1}^2}{q_{1}} \right)=\left(\frac{x+1}{q_{1}} \right)=\left(\frac{x-1+2}{q_{1}} \right)= \left(\frac{2}{q_{1}} \right)=-1,
\end{align*}
which is impossible. Therefore, $d_{2}=p$ and $\sqrt{p \eps} = py_{2}+y_{1}\sqrt{D}=\gamma_{1}$, where $\gamma_{1}$ has norm $-p$ in $\Q$.\\

Let  $\epsp = r + s \sqrt{2D}$, where $r$ is odd and $s$ is even. Like before, we have 
\begin{equation*}
    \begin{cases}
        r \pm 1 = d_{1} s_{1}^{2}, \\
        r \mp 1 = 2d_{2} s_{2}^{2},
    \end{cases}
\end{equation*}
with $d_{1}d_{2}=D$ and $s_{1}s_{2}=s$. The cases $d_{2}=1$ or $D$ cannot happen again due to \cite{azizi2000}. This implies that $\sqrt{\epsp} \not \in K_{1}$, since $\sqrt{d_{1} \epsp} \in K_{1}$ for some proper divisor $d_{1}$ of $D$. In particular, $K_{1}(\sqrt{\epsp}) = K_{1}(\sqrt{\eps})$ if and only if $d_{1}$ is equal to $p$ or $q_{1}q_{2}$. If $d_{1}=q_{1}q_{2}$, then $\left(\frac{q_{1}q_{2}}{p} \right)= \left(\frac{\pm 2}{p} \right)=-1$, which is absurd due to our conditions on quadratic symbols. Thus, $d_{1}=p$. 
\\

If
\begin{equation*}
    \begin{cases}
        r - 1 = p s_{1}^{2}, \\
        r + 1 = 2q_{1}q_{2} s_{2}^{2},
    \end{cases}
\end{equation*}
then $\left(\frac{p}{q_{1}} \right)= \left(\frac{- 2}{q_{1}} \right)=1$, which is a contradiction. Therefore,
\begin{equation*}
    \begin{cases}
        r + 1 = p s_{1}^{2}, \\
        r - 1 = 2q_{1}q_{2} s_{2}^{2},
    \end{cases}
\end{equation*}
is the only possibility, and $\sqrt{\eps \epsp} \in K_{1}$ if only if $p(r+1)$ is a square.
\\

Let $w = s_{1}\sqrt{p}+s_{2}\sqrt{2q_{1}q_{2}}$. Then, as before,  $\sqrt{2p\epsp} = \sqrt{p} \cdot w = ps_{1}+s_{2}\sqrt{2D}=\gamma_{2}$, where the norm of $\gamma_{2}$ is equal to $2p$. Multiplying these equalities, we obtain $\sqrt{2}p \cdot\sqrt{\eps \epsp} =\gamma_{1} \gamma_{2}$ and the result on the norms follows.
\end{proof}
\vspace{0.5cm}

The following result is derived from (\cite{kumakawa}, Lemma 2.1) and proved in (\cite{laxmi}, Lemma 5.1). Even though the authors prove it for the current case, it holds with more generality. Consider the following diagram
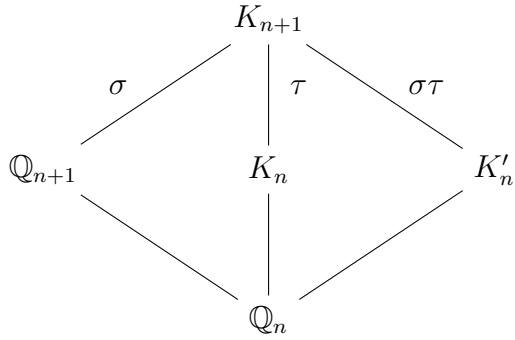
\begin{figure}[h] 
\centering
 \begin{tikzpicture}

    \node (Q1) at (0,0) {$\Q_n$};
    \node (Q2) at (3,2) {$K_{n}^{\prime}$};
    \node (Q3) at (0,2) {$K_n$};
    \node (Q4) at (-3,2) {$\Q_{n+1}$};  
    \node (Q5) at (0,4) {$K_{n+1}$};
    
    \draw (Q1)--(Q2);
    \draw (Q1)--(Q3); 
    \draw (Q1)--(Q4);
    \draw (Q2)--(Q5);
    \draw (Q3)--(Q5);
    \draw (Q4)--(Q5);
    
    \node (R1) at (2.1,3.1) {$\sigma\tau$};
    \node (R2) at (0.4, 3.1) {$ \tau $};
    \node (R3) at (-2, 3.1) {$ \sigma $};
     \end{tikzpicture}
    \end{figure}

where both $\sigma$ and $\tau$ are generators of the corresponding Galois groups. Then we have the following:

\begin{lemma}\label{inequality of intermediate fields}
Let $K = \Q(\sqrt{D})$ with $D \equiv 5 \pmod 8$. Then 
\begin{align*}
    \al A_{n+1}\ar \leq \dfrac{\al A_{n}^{\prime}\ar \cdot \al A_{n} \ar}{2 \cdot \left[ E(K_n): N_{K_{n+1}/K_{n}}(E(K_{n+1})) \right]} \leq \frac{1}{2} \cdot \al A_n \ar \cdot\al A_{n}^{\prime}\ar 
\end{align*}
\end{lemma}

\begin{proof}
The proof is analogous to (\cite{laxmi}, Lemma 5.1) since the same steps can be taken whenever the prime $2$ is inert in the extension $K / \Q$.\\ 

The authors proved that $ \al A_{n+1} \ar \; \leq \frac{1}{2} \cdot \al A_{n}^{\prime} \ar \cdot \al A_{n+1}^{\tau +1} \ar$, where $A_{n+1}^{\tau +1}=\{ [\mathfrak{a}]^{\tau}\cdot [\mathfrak{a}]: [\mathfrak{a}] \in A_{n+1} \}$. The first inequality comes from the fact that $A_{n+1}^{\tau +1}$ is contained in the group $B_{n+1}^{\langle \tau \rangle}$ of classes containing an ideal fixed by the action of $\tau$ (strongly ambiguous classes), and the extension $K_{n+1}/K_{n}$ is totally ramified. Hence, due to the genus formula \cite{genusformula}, the following inequality holds:
\begin{align*}
\al A_{n+1}^{\tau +1} \ar  \leq \al B_{n+1}^{\langle \tau \rangle}\ar = \frac{\al A_{n} \ar}{\left[ E(K_n): N_{K_{n+1}/K_{n}}(E(K_{n+1})) \right]}
\end{align*}

Piecing everything together, the result follows.
\end{proof}
\vspace{0.5cm}

We will try to determine an F.S.U of the field $K_{2}=\Qextbi{\sqrt{2+\sqrt{2}},\sqrt{D}}$ in terms of the units of the intermediate fields in the previous diagram when $n=1$. Using Lemmermeyer's generalization \cite{lemmermeyer} and \Cref{k1hatsecondcase}, the index  $q(K_{2})=[E(K_{2}) : E(K_{1}) \cdot E(K_{1}^{\prime}) \cdot E(\Q_{2})]$ satisfies the equation
\begin{equation} \label{lemmermeyer}
    \al A_{2} \ar = \frac{1}{8} \cdot q(K_{2})  \cdot \al A_{1} \ar \cdot \al A_{1}^{\prime} \ar = \frac{q(K_{2})}{2} \cdot \al A_{1} \ar.
    \end{equation}

First, we notice that if $x$ is $K_{1}=\Qext{\sqrt{2}, \sqrt{D}}$, the norm from $K_{2}$ to any of the other  subfields is equal to $N_{K_{1}/\Q_{1}}(x)$, since the other automorphisms change $\sqrt{D}$. The same argument works if $x$ is in the other subfields. Due to Wada \cite{wada}, the units of $E(K_{2})$ are generated by the elements of $E(K_{1}) \cdot E(K_{1}^{\prime}) \cdot E(\Q_{2})$ that are squares in $K_{2}$ and the units of the intermediate subfields. To see this, if $u \in  E(K_{2})$, then $$N_{K_{2}/K_{1}}(u) \cdot N_{K_{1}/K_{1}^{\prime}}(u) \cdot N_{K_{2}/\Q_{2}}(u)= u^2 \cdot N_{K_{2}/\Q_{1}}(u).$$

The units of  $\Q_{2}$ are determined by the cyclotomic units \cite{washington}
\begin{align*}
    \xi_{a}=\frac{\sin{(\pi a/16)}}{\sin{(\pi/16)}}, \; a \in \{3, 5, 7 \}
\end{align*}

We can see that $\xi_{5}=1+\sqrt{2}+\sqrt{2+\sqrt{2}}$ has norm $\epst$ in $\Q_{1}$ and $\xi_{7}=1+\sqrt{2}+\sqrt{2}\sqrt{2+\sqrt{2}}$ has norm $-1$. Since $\xi_{3} \cdot \xi_{5} \cdot \xi_{7}^{-1} = \epst$, we can choose a F.S.U. $\{w_{1}, w_{2}, \epst\}$ of $\Q_{2}$, such that $N_{\Q_{2}/\Q_{1}}(w_{1})=\epst$ and $N_{\Q_{2}/\Q_{1}}(w_{2})=-1$.
\\

In the extension $K_{1}^{\prime}/\Q_{1}$, as seen in \Cref{k1hatsecondcase}, $\epst$ is not a norm. In particular, if $v$ is a unit with $N_{K_{1}^{\prime}/\Q_{1}}(v)=\pm \epst^n$, then $n$ has to be even, otherwise $v \cdot \epst^{-(n-1)/2}$ has norm $\epst$. Since $\epst$ forms part of the fundamental system of $E(K_{1}^{\prime})$ (it belongs to the three subfields and $\sqrt{\epst} \not \in K_{2}$), the unit $v$ can be replaced by  $v \cdot \epst^{-n/2}$, so that its norm is $\pm 1$. In particular, we can choose a F.S.U. of $E(K_{1}^{\prime})$ of the form $\{v_{1}, v_{2}, \epst \}$, where $N_{K_{1}^{\prime}/\Q_{1}}(v_{1})=1$, and $N_{K_{1}^{\prime}/\Q_{1}}(v_{2})=\pm 1$, depending on whether $-1$ is a norm or not.
\\

Now assume we are in the case of \Cref{doublesymbol-1} where $\sqrt{\eps \epsp} \not \in K_{1}$. Then the fundamental system of $K_{1}$ is $\{ \eps, \epsp, \epst\}$. The F.S.U. of $E(K_{1}^{\prime})$ is of the form $\{v_{1}, v_{2}, \epst \}$, where $N_{K_{1}^{\prime}/\Q_{1}}(v_{1})=1$, and $N_{K_{1}^{\prime}/\Q_{1}}(v_{2})= -1$. The one for $E(\Q_{2})$
 is $\{w_{1}, w_{2}, \epst \}$, where $N_{\Q_{2}/\Q_{1}}(w_{1})=\epst$, and $N_{\Q_{2}/\Q_{1}}(w_{2})=-1$. Choosing the fundamental systems this way, we prove the following result, which gives a new family of examples:

\begin{theorem} \label{doublesymbol-1part2}
    Let $K = \Qext{\sqrt{D}}$ where $D=pq_{1}q_{2}$, such that $p \equiv 5 \mod{8}$, \\ $q_{1}, q_{2} \equiv 3 \mod{8}$, and $\left(\frac{p}{q_{1}}\right) =\left(\frac{p}{q_{2}}\right) = -1$. Furthermore, suppose that $\sqrt{\eps \epsp} \not \in K_{1}$. Then $X_{\infty} \cong A_{1}$.
\end{theorem}
\begin{proof}
We will show that $\al A_{2} \ar = \al A_{1} \ar$. From \Cref{k1hatsecondcase}, we must have $\al A_{1}^{\prime} \ar=4$. Thus, using \Cref{inequality of intermediate fields}, we have
\begin{align*}
     \al A_{2} \ar \leq \dfrac{2 \cdot \al A_{1} \ar}{\left[ E(K_1): N_{K_{2}/K_{1}}(E(K_{2})) \right]} \leq 2 \cdot \al A_{1} \ar .
\end{align*}

As mentioned in \Cref{lemmermeyer}, we have $\al A_{2} \ar = \frac{q(K_{2})}{2} \cdot \al A_{1} \ar$, which implies that $q(K_{2})$ is at most $4$; i.e., at most two independent elements of $E(K_{1}) \cdot E(K_{1}^{\prime}) \cdot E(\Q_{2})$ are squares in $K_{2}$.
\\

Suppose $u \in E(K_{2})$ satisfies $u^{2} = \eps ^{a} \; \epsp^{b} \; \epst^{c} \; v_{1}^d \; v_{2}^e \; w_{1}^f \; w_{2}^g .$ Applying the norm to $K_{1}$ on both sides of the equation, we have
\begin{align*}
    N_{K_{2}/K_{1}}(u)^{2} = \eps ^{2a} \; \epsp^{2b} \; \epst^{2c}  \; (-1)^{e} \; \epst^f \; (-1)^g .
\end{align*}

Since $\epst$ is not a square in $K_{1}$, it follows that $e \equiv g \mod{2}$, and $f \equiv 0 \mod{2}$. Hence,
\begin{align*}
    u^{2} = \eps ^{a} \; \epsp^{b} \; \epst^{c} \; v_{1}^d \; v_{2}^e \; w_{2}^e \; t^{2} , \; \text{for some} \; t.
\end{align*}

We can omit $t$, since it can be included inside $u$ and the norm to $K_{1}$ of any square we leave out of the equations always lies in $\Q_{1}$. By taking the norm to $\Q_{2}$, we can see that \begin{align*}
    N_{K_{2}/\Q_{2}}(u)^{2} = \epst^{2c}  \; (-1)^{e} \;  w_{2}^{2e},
\end{align*}
which implies $e \equiv 0 \mod{2}$. Hence, modulo squares, we have
\begin{align*}
     u^{2} = \eps ^{a} \; \epsp^{b} \; \epst^{c} \; v_{1}^d.
\end{align*}

Suppose $N_{K_{2}/K_{1}}(E(K_{2})) = E(K_{1})$, so that both $\eps$ and $\epsp$ are norms from $E(K_{2})$. If $N_{K_{2}/K_{1}}(u) = \epsp$, then we can write
\begin{align*}
     u^{2} = \eps ^{a} \; \epsp^{b} \; \epst^{c} \; v_{1}^d \; t^2,
\end{align*}
where the norm of $t$ is some element $w$ in $\Q_{1}$. Applying the norm to $K_{1}$ on both sides, we get \begin{align*}
     \eps^{2} = \eps ^{2a} \; \epsp^{2b} \; \epst^{2c} \; w^{2},
\end{align*}
from where we conclude that $b = 0$ and $a = 1$. Consequently, we can assume that for some element $u \in E(K_2)$ we have $u^2 = \eps \; \epst^{c} \; v_{1}^{d}$ for some $c, d \in \{0, 1\}$. Similarly, if $N_{K_{2}/K_{1}}(v)=\epsp$, we can assume $v^2=\eps \; \epst^{m} \; v_{1}^{n}$ for some $m, n \in \{0 , 1\}$ and $v \in E(K_2)$. 
\\

First, we notice that we cannot have mixed terms inside the square root since this would give us $3$ independent square roots by taking the product. For example, if $\sqrt{\eps v_{1}}$ and $\sqrt{\epsp v_{1}}$ are in $K_{2}$, then $\sqrt{\eps\epsp} \in K_{2}$. We also recall that $K_{2}(\sqrt{\eps}) = K_{2}(\sqrt{d_{1}})$ for some proper divisor $d_{1}$ of $D$ (\Cref{doublesymbol-1}). Thus, we cannot have $\sqrt{\eps} \in K_{2}$. Similarly, both $\sqrt{\epsp}$ and $\sqrt{\eps\epsp}$ are not in $K_{2}$.
\\

Therefore, either $\sqrt{\eps \epst} \in K_{2}$ or $\sqrt{\epsp \epst} \in K_{2}$. Consider the map $\theta$ in $\text{Gal}(K_{2}/\Q)$ defined by
\begin{align*}
    \theta \bl \sqrt{2+\sqrt{2}} \br =\sqrt{2-\sqrt{2}} \; , \; \theta(\sqrt{D} ) = \sqrt{D}.
\end{align*}

Suppose $\sqrt{\epsp \epst}=u \in K_{2}$. Since $\theta(\sqrt{2})=-\sqrt{2}$, we can deduce that $\theta(\epsp)$ is its conjugate over $\Q$. 
Thus, \begin{align*}
    (u \cdot \theta(u))^2= N_{\kbar / \Q} (\epsp) \cdot N_{\Q_{1} / \Q} (\epst) = -1,
\end{align*}
which is impossible. Similarly, if $\sqrt{\eps \epst}=u \in K_{2}$, then
\begin{align*}
    (u \cdot \theta(u))^2= -(\eps)^{2} < 0.
\end{align*}

Hence, at most one of these elements is a norm and $\left[ E(K_1): N_{K_{2}/K_{1}}(E(K_{2})) \right]>1$, which implies that $\al A_{2} \ar \leq \al A_{1} \ar$ due to \Cref{inequality of intermediate fields}. The result follows for this case using \Cref{fukuda stability}.
\end{proof}

\begin{remarkref}
    The same proof does not work when $\sqrt{\eps \epsp \epst} \in K_{1}$ since the norms to the subfields can be positive (\Cref{doublesymbol-1}), giving no contradiction in the last steps of the previous result. For example, for $D=5 \cdot 3 \cdot 163$ we have $X_{\infty} \cong A_{1}$ and for $D=5 \cdot 3 \cdot 643$ we have $X_{\infty} \cong A_{2}$. There is no clear pattern for when the chain stops but we believe that some conditions can be drawn using the quadratic symbols over a different ring, like it was done at the end of the previous section.
\end{remarkref}

\section{Third Case}

Throughout this section, we assume $K=\Qext{\sqrt{D}}$, where $D=pq$ with $p \equiv 5 \mod 8, \;$ \\ $ q \equiv 1 \mod 8$, and $\left( \dfrac{2}{q} \right)_{4} \neq (-1)^{\frac{q-1}{8}}$. We provide new examples of infinite families of real quadratic fields satisfying Greenberg's conjecture. The following corollary is only mentioned in \cite{mouhib}, so we give a short proof:

\begin{corollary} \label{mohib corollary}
    Suppose $X_{\infty}$ is cyclic and $A_{n_{0}} \neq 0$. Then $X_{\infty} \cong A_{n_{0}}$ if there exists a non-trivial element of $A_{n_{0}}$ which capitulates in $A_{n_{0}+1}$.
\end{corollary}
\begin{proof}
 Using the notation from (\cite{mouhib}, Proposition 2.6), we have $N = n_{0}$, and there exists a non-trivial element of $A_{N}$ which capitulates in $A_{N+1}$. If $|A_{N}|=p^{r}$, then $A_{N+1} \cong X_{\infty}/p^r X_{\infty}.$ Since $A_{N+1}$ is cyclic, we have $\al A_{N+1} \ar \leq \al A_{N}\ar = p^r$. The extension is totally ramified, which implies that $\al A_{N} \ar \leq \al A_{N+1}\ar $. We can conclude that $X_{\infty} \cong A_{N} = A_{n_{0}}$, due to \Cref{fukuda stability}.
\end{proof}

Following the discussion in \cite{mouhib} after Theorem 4.1, we have the following lemma: 

\begin{lemma} \label{capitulationpq}
The following statements hold:
\begin{enumerate}[i-]
    \item Assume that the ideal $Q$ above the prime $q$ in $K$ is not principal.  Then $X_{\infty} \cong A_{0}$.
    \vspace{0.2cm}
    
  \item Assume that the ideal $Q^{\prime}$ above the prime $q$ in $K^{\prime}$ is not principal, and $q \equiv 1 \mod {16}$. Then $X_{\infty} \cong A_{1}$.
\end{enumerate}
\end{lemma}

\begin{proof}
    
Since $n_{0}=0$ for $K$, $n_0=1$ for $K^{\prime}$, and all the $2$-groups involved are non-trivial, we can apply the previous corollary with $N=0$ or $N=1$, once we prove that a non-trivial ideal capitulates in the next level.
\\

Let $Q$ denote the prime of $K$ above $q$, and assume that it is not principal. Since $q \equiv 1 \mod{8}$, it splits in $\Q_1=\Q(\sqrt{2})/\Q$ and ramifies in $K/\Q$. Thus, $Q$ splits into two primes $Q_1, Q_2$ in $K_1$. These primes cannot be principal, because they are conjugate by the non-trivial automorphism of $\text{Gal}(K_{1}/K)$, so if $Q_1 = \langle \alpha \rangle$, it follows that $Q = \langle N_{K_1/K}(\alpha) \rangle$ in $K$, which is a contradiction due to our assumption. Given that all the $q$-adic primes ramify in  $K_{1}/\Q_{1}$, it follows that $Q_{i}^{2}$ comes from an ideal of $\Q_{1}$ for $i=1,2$. Since the class group of $\Q_{1}$ is trivial, it follows that $Q_{i}^2$ is principal in $K_{1}$. Any cyclic group has a unique element of order 2, therefore $Q_{1}$ and $Q_{2}$ lie within the same class of $A_{1}$. Hence, $Q = Q_{1} Q_{2}$ is principal in $K_{1}$ and $Q$ capitulates in $K_{1}$. The result then follows for the first statement.
\\ 

Let $Q^{\prime}$ be the prime above $q$ in $K^{\prime}$. Assume $Q^{\prime}$ is not principal and $q \equiv 1 \mod{16}$. Given that $q$ splits in $\Q_{1}$, it follows that $Q^{\prime}$ splits into two non-principal prime ideals $Q_{1}^{\prime}, \; Q_{2}^{\prime}$ of $K_{1}$. The field $\Q_{2}=\Qext{\sqrt{2+\sqrt{2}}}$ is contained in $\Qext{\zeta_{16}}$.
Since $q \equiv 1 \mod{16}$, it splits completely in the extension $\Qext{\zeta_{16}}/\Q$ \cite{washington}. Therefore, both primes $Q_{1}^{\prime}$ and $Q_{2}^{\prime}$ are non-trivial and split in $K_{2}$. Again, all the $q$-adic primes are ramified in $K_{2}/\Q_{2}$ and the group $A_{2}$ is cyclic. Hence, both primes capitulate in $K_{2}$ due to a similar argument. Once again, we can conclude that $X_{\infty} \cong A_{1}$ from the previous result.
\end{proof}

\begin{remarkref}
The case $q \equiv 9 \mod{16}$ turns out to be more complicated since the primes over $q$ in $\Q_{1}$ become inert in $\Q_{2}$. Some work has been done with similar cases using circular units \cite{kumakawa} to analyze the units of the field $K_{1}^{\prime}$ or using Iwasawa theory \cite{fukudasimilarcase}. Still, the modular conditions make it hard to get similar results. We will refer to this case later in this section.
\end{remarkref}

\begin{proposition} \label{principality}
    Let $N$ denote the norm map of the corresponding field to $\Q$. Then the following hold:
    \begin{enumerate} [i-]
        \item If $N(\eps)=-1$, then the ideal above $q$ in $K$ is not principal.
        \item If $N(\epsp)=-1$, then the ideal above $q$ in $K^{\prime}$ is not principal.
    \end{enumerate}
\end{proposition}
\begin{proof}
First, assume that $N(\eps)=-1$. Let $Q$ denote the prime of $K$ above $q$. Assume by contradiction that $Q$ is principal. Then there exists a fractional ideal $\langle x \rangle$ in $K$ such that $Q=\langle x \rangle$. By squaring, we get $q=x^2 \cdot \eps^n$ for some $n$. Applying the norm map on both sides, it follows that $ q^2= N_{K/\Q}(x)^2 (-1)^n$. Hence, $n$ is even and $q=(|N_{K/\Q}(x) | \cdot  \eps^{\frac{n}{2}})^2$ which is a contradiction. It follows that $Q$ is not principal. The second case is analogous.
\end{proof}

Using the classification in \Cref{scholz}, we have the following corollary:

\begin{corollary} \label{pqcase} Suppose $K= \Q (\sqrt{pq})$ with the same conditions as before. Then the following statements hold:
    \begin{enumerate}[i-]
\item Assume $\left( \dfrac{p}{q} \right)=-1$. Then $X_{\infty} \cong A_{0} \cong C_{2}$.  
\smallskip
\item Assume $\left( \dfrac{p}{q} \right)=1$ and $\left( \dfrac{p}{q} \right)_{4}=\left( \dfrac{q}{p} \right)_{4}=-1$. Then $X_{\infty} \cong A_{0} \cong C_{4}$.  
\smallskip
\item Assume $\left( \dfrac{p}{q} \right)=1$ and $\left( \dfrac{p}{q} \right)_{4}=\left( \dfrac{q}{p} \right)_{4}=1$. If the ideal above $q$ in $K$ is not principal, then $X_{\infty} \cong A_0$. In particular, the result holds whenever $N(\eps)=-1$.
\smallskip
\item Assume $\left( \dfrac{p}{q} \right)=1$, $\left( \dfrac{p}{q} \right)_{4}=-\left( \dfrac{q}{p} \right)_{4}$ and $q \equiv 1 \mod {16}$. If the ideal above $q$ in $K^{\prime}$ is not principal, then $X_{\infty} \cong A_1$. In particular, the result holds whenever $N(\epsp)=-1$.
\end{enumerate}
\end{corollary}
\smallskip

\begin{proof}
    This follows from \Cref{mohib corollary}, \Cref{capitulationpq}, and \Cref{scholz}.
\end{proof}    

For the third case in the previous proposition we notice the following:

\begin{proposition} \label{bothprincipal}
    Assume $\left( \dfrac{p}{q} \right)=1$ and $\left( \dfrac{p}{q} \right)_{4}=\left( \dfrac{q}{p} \right)_{4}=1$. Then $N(\epsp)=1$ and the ideal above $q$ is principal in $\kbar$.
\end{proposition}

\begin{proof}
    By Kaplan (\cite{kaplan}, §4, Proposition B2), since $\left( \dfrac{2p}{q} \right)_{4} \neq \left( \dfrac{q}{2p} \right)_{4}$, the principal genus contains both $[q, 0, -2p]$ and $[-q, 0, 2p]$, and only one of them is equivalent to the unit. In particular, the form $[-1, 0, D]$ is not equivalent to the unit form and $N(\epsp)=1$. We conclude as before that the ideal over $q$ in $K^{\prime}$ is principal. Thus $q = x^2 \cdot \epsp^n$ for some $n$. Since $\sqrt{q} \not \in K_1$, $n$ must be odd, $K_{1}(\sqrt{\epsp})=K_{1}(\sqrt{q})$.
\end{proof}

We mention some results and consequences in terms of fundamental units and class groups:

\begin{corollary} \label{corol1} Suppose $\left( \dfrac{p}{q} \right)=-1$. Then the 2-class group $A_{0}^{\prime}$ of $K^{\prime}$ is isomorphic to $C_2 \times C_2$. Furthermore, $
\{ \eps, \epsp, \epst \}$ is a fundamental system of units F.S.U. of $K_{1}$.
\end{corollary}

\begin{proof}
By Kaplan (\cite{kaplan}, §4, Proposition $A_{2}$), we have $N(\epsp)=-1$, as only the unit and its inverse are in the principal genus. Thus, the $4$-rank is $0$ and $A_{0}^{\prime} \cong C_2 \times C_2$, since it must be equal to its narrow class group. From \Cref{kubotakuroda formula} and \Cref{pqcase}, we have $\al A_{1}\ar =\al A_{0} \ar = q(K_{1}) \cdot \al A_0\ar$. Therefore, $q(K_{1})=1$, and the fundamental system is the one indicated.
\end{proof}

\begin{corollary} \label{corol2} Assume $\left( \dfrac{p}{q} \right)=1$. Suppose we are in one of the following cases:
\begin{enumerate}[i-]
    \item $\left( \dfrac{p}{q} \right)_{4}=\left( \dfrac{q}{p} \right)_{4}=-1$.
    \item $\left( \dfrac{p}{q} \right)_{4}=\left( \dfrac{q}{p} \right)_{4}=1$, and the ideal above $q$ in $K$ is not principal.
\end{enumerate}
\smallskip
Then, the 2-class group $A_{0}^{\prime}$ of $K^{\prime}$ is isomorphic to $C_2 \times C_2$ and the norm of the fundamental unit $\epsp$ is $1$. Furthermore, $\{ \eps, \epsp, \epst \}$ is a F.S.U. of $K_{1}$. 
In the second case, if the ideal above $q$ in $K$ is principal, then $\{ \sqrt{\eps \epsp}, \epsp, \epst \}$ is a F.S.U. of $K_{1}$.
\end{corollary}

\begin{proof}
From \Cref{kubotakuroda formula} we have 
\begin{align*}
\al A_1 \ar =\frac{q(K_{1}) \cdot \al A_0 \ar \cdot \al A_{0}^{\prime} \ar}{4}.
\end{align*}

Due to \Cref{pqcase}, $\al A_1\ar =\al A_0 \ar$. Therefore,  $q(K_{1}) \cdot \al A_{0}^{\prime} \ar = 4$, which implies $\al A_{0}^{\prime} \ar  \leq 4$. Furthermore, the $2$-rank of this group is $2$, from where it follows that $A_{0}^{\prime} \cong C_{2} \times C_{2}$, $q(K_{1})=1$, and the fundamental system is the one stated. The R\`edei-Reichardt matrix \cite{redei} for $K^{\prime}$ is

   \begin{center}
 $\begin{pmatrix}
1 & 1 & 0\\
1 & 1 & 0 \\
0 & 0 & 0 \\
\end{pmatrix}$.
\end{center}

Hence, the $4$-rank of the narrow class group of $K^{\prime}$ is $3-1-1=1$ \cite{redei}. Since the narrow class number differs from the strict class number, the fundamental unit must have norm equal to $1$. Finally, if the ideal above $q$ is principal in $K$, it follows (\Cref{bothprincipal}) that $K_{1}(\sqrt{\eps})=K_{1}(\sqrt{q})=K_{1}(\sqrt{\epsp})$ and $\sqrt{\eps \epsp} \in K_{1}$.
\end{proof}

For the last case in \Cref{pqcase}, we have the following results:

\begin{proposition}\label{normequivalence}
Suppose $\left( \dfrac{p}{q} \right)=1$ and $\left( \dfrac{p}{q} \right)_{4}=-\left( \dfrac{q}{p} \right)_{4}$. Then the ideal over $q$ is not principal in $K^{\prime}$ if and only if $N(\epsp)=-1$.
\end{proposition}

\begin{proof}
    One of the implications was already discussed in \Cref{principality}. Assume that the ideal over $q$ is not principal in $K^{\prime}$. From Kaplan (\cite{kaplan}, Proposition $B_{2}$), the forms in the principal genus are $f=[1, 0, -pq]$, $\overline{f}=[-1, 0, pq]$, $h=[q, 0, -2p]$ and $\overline{h}=[-q, 0, 2p]$. We cannot have $f$ equivalent to $h$ or $\overline{h}$, otherwise $q$ would be principal in $K^{\prime}$. Therefore, $f$ is equivalent to $\overline{f}$ and $f$ represents $-1$, which completes the result. 
\end{proof}

\begin{lemma} \label{lastcase1mod16}
Suppose $\left( \dfrac{p}{q} \right)=1$ and $\left( \dfrac{p}{q} \right)_{4}=-\left( \dfrac{q}{p} \right)_{4}$. If $q \equiv 1 \mod{16}$, then the following statements hold:
\begin{enumerate} [i-]
    \item
        If $N(\epsp)=-1$, the fundamental system of units of $K_{1}$ is $\{ \eps, \epsp, \epst \}$ and we get \\ $\al A_{1} \ar =\frac{1}{2} \cdot \al A_{0}^{\prime} \ar .$ Moreover, if $\left( \dfrac{2p}{q} \right)_{4}=-1$, then $N(\epsp)=-1$ and $\al A_{1} \ar=4$.
        
    \smallskip

\item If $N(\epsp)=1$, then $\{ \sqrt{\eps \epsp}, \eps, \epst \}$ is a F.S.U. of $K_{1}$. Hence, $q(K_{1})=2$ in this case.

\end{enumerate}

\end{lemma}

\begin{proof}
For the field $K$, the norm of $\eps$ is 1 and one of the ideals above $p$ or $q$ in $K$ is principal (\cite{scholz}, \textbf{cf.} \cite{kaplandiv}). Thus, we have $K_{1}(\sqrt{p})=K_{1}(\sqrt{q})=K(\sqrt{\eps})$. In particular, $\sqrt{\eps} \not \in K_1$. \\

If $N(\epsp)=-1$, the only possible system is $\{ \eps, \epsp, \epst \}$. Therefore $q(K_{1})=1$, and since $\al A_0 \ar=2$, we get $ \al A_1 \ar=\frac{1}{2} \cdot \al A_{0}^{\prime} \ar.$ In the case when $\left( \dfrac{2p}{q} \right)_{4}=-1$, due to (\cite{kaplan}, §4, Proposition $A_{2}$), the $8$-rank of the narrow class group of $K^{\prime}$ is $0$ and $N(\epsp)=-1$. Since the $4$-rank of the narrow class group is $1$ and the $2$-rank is $2$, we have $A_{0}^{\prime} \cong C_4 \times C_2$.
\\

Finally, if $N(\epsp)=1$ we must have $K_{1}(\sqrt{\epsp}) = K_{1}(\sqrt{q}) = K_{1}(\sqrt{\eps})$ due to \Cref{normequivalence}.
\end{proof}

We prove a result similar to \Cref{k1hatsecondcase}, that does not need any extra conditions on the quadratic or biquadratic symbols, that will  give us an idea of what we can do next:

\begin{lemma} \label{fieldK1hat}
    Suppose $p \equiv 5 \mod{8}$ and $q \equiv 1 \mod{8}$ with $\left( \frac{2}{q} \right)_{4} \neq (-1)^{\frac{q-1}{8}}$. Then the 2-class group of $K_{1}^{\prime}=\Qext{\sqrt{(2+\sqrt{2})pq}}$ is isomorphic to $ C_2 \times C_2$. In particular, $\al A_{1}^{\prime} \ar=4$.
\end{lemma}

\begin{proof}
    We first show that the 2-rank of $A_{1}^{\prime}$ is $2$. Consider the extension $K_{1}^{\prime} / \Q_{1}$. Since there are $4$ ramifying primes in this extension (2 over $q$, 1 over $p$ and 1 over $2$), we have:
    \begin{align*}
        r_{2}(A_{1}^{\prime})=3-e(K_{1}^{\prime} / \Q_1)
    \end{align*}
where $e(K_{1}^{\prime} / \Q_1)$ is the unit index, as mentioned before. Since the extension $K_{1}^{\prime} (\sqrt{p}, \sqrt{q})$ is unramified and biquadratic over $K_{1}^{\prime}$, it follows that the $2$-rank of the class group is at least $2$.
\\

If $\mathcal{L}$ is not a 2-adic prime of $\Q_1$, then the extension $\Q_2 /\Q_1$ is unramified at $\mathcal{L}$. Since $\Q_{2}=\Q_{1}(\epst \sqrt{2})$, the Hilbert symbol satisfies
\begin{align*}
    \left( \dfrac{\epst, \epst \sqrt{2}}{L} \right)=1.
\end{align*}

Let $\mathcal{Q}$ be an ideal over $q$ in $Q_{2}$. Then as discussed in \cite{azizi2001}, 
\begin{align*}
    \left( \dfrac{\epst,q}{\mathcal{Q}} \right) = \left( \frac{2}{q} \right)_{4} \cdot  \biggl( \frac{q}{2} \biggr)_{4}=-1.
\end{align*}
Thus,
\begin{align*}
    \left( \dfrac{\epst,pq\epst \sqrt{2}}{\mathcal{Q}} \right)=-1
\end{align*}

Hence, we know that $\epst$ is not a norm from $K_{1}^{\prime}$. It follows that $e(K_{1}^{\prime}/\Q_1)=1$ and the $2$-rank of $A_{1}^{\prime}$ is exactly $2$. From here the proof is analogous to \Cref{k1hatsecondcase}, by simply considering the intermediate fields  of the extension $M=K_{1}^{\prime} (\sqrt{p}, \sqrt{q})= Q(\sqrt{2+\sqrt{2} }, \sqrt{p}, \sqrt{q}) $ over $K_{1}^{\prime}$.
\end{proof}

\vspace{0.2cm}

We write down some final remarks on the cases that are left in (\cite{mouhib}, Theorem 3.8). One of the reasons it seems hard to understand what happens in some of the cases is that we cannot replicate what was done in \Cref{doublesymbol-1part2}. For instance, in \Cref{doublesymbol-1}, whenever $\sqrt{\eps \epsp \epst} \in K_{1}$, the norm taken to the subfields can be positive or negative, and the same argument cannot be applied unless we have a real understanding of the units of $K_{2}$. The behavior is comparable to what happens in some of the cases \Cref{caseqequal5}, where the module is known to be finite using Iwasawa theory, but it is not easy to predict where the chain stops growing. For one of the fields left out in the third case, we conjecture that the following statement is true:
\begin{conj}
    Let $K=\Qext{\sqrt{D}}$, where $D=pq$ with $p \equiv 5 \pmod 8, \; q \equiv 1 \pmod 8$, and $\left( \dfrac{2}{q} \right)_{4} \neq (-1)^{\frac{q-1}{8}}$. Assume that $\left( \dfrac{p}{q} \right)=1$, $\left( \dfrac{p}{q} \right)_{4}=-\left( \dfrac{q}{p} \right)_{4}, \; q \equiv 9 \mod {16}$, and $N(\epsp)=-1$. Then $X_{\infty} \cong A_{1}$.
\end{conj}

We found no evidence to reject this conjecture but the methods we have at our disposal were not enough to prove the result. A better understanding of the units of $K_{2}$, using the units of the intermediate fields and the quadratic symbols of $\Z[\sqrt{2},i]$, might lead to a positive result in the future.

\subsection*{Acknowledgments} The results presented in this document are from the author's Ph.D. thesis dissertation at University of Maryland, and this article is part of project 821-C4-257 at Universidad de Costa Rica. First, I would like to thank my advisor, Professor Niranjan Ramachandran, for his encouragement and constant help during my years at University of Maryland. I would also like to thank Professor Lawrence C. Washington for always keeping his door open to me and for all the ideas and clarifications he gave me on the elaboration of my thesis, this would not have been possible without their help and for that, I will always be grateful.  Universidad de Costa Rica funded this research, I would like to thank the Department of Mathematics, the "Vicerrector\'ia de Investigaci\'on" for funding the project, and the center CIMPA (Centro de investigaci\'on de Matem\'atica Pura y Aplicada).

\nocite{azizi2000}

\end{document}